\documentclass[11pt]{amsart}
\usepackage{amsmath, amsthm, amssymb}
\usepackage{graphicx}
\usepackage{mathbbol}
\usepackage{txfonts}
\usepackage[usenames]{color}
\usepackage{enumerate}
\usepackage{xcolor}

\newtheorem{thm}{Theorem}[section]
\newcommand{\bt}{\begin{thm}}
\newcommand{\et}{\end{thm}}

\newtheorem{cor}[thm]{Corollary}   

\newcommand{\bc}{\begin{cor}}
\newcommand{\ec}{\end{cor}}

\newtheorem{lem}[thm]{Lemma}   

\newcommand{\bl}{\begin{lem}}
\newcommand{\el}{\end{lem}}

\newtheorem{prop}[thm]{Proposition}
\newcommand{\bp}{\begin{prop}}
\newcommand{\ep}{\end{prop}}

\theoremstyle{definition}
\newtheorem{defn}[thm]{Definition}

\newcommand{\bd}{\begin{defn}}    
\newcommand{\ed}{\end{defn}}

\newtheorem{rmrk}[thm]{Remark}   
  
\newtheorem*{notation}{Notation}

\newcommand{\br}{\begin{rmrk}}
\newcommand{\er}{\end{rmrk}}

\newcommand{\GHto}{\xrightarrow{\>\mathrm{GH}\> }}

\newcommand{\IFto}{\xrightarrow{\;\;\mathrm{F}\;\; }}

\newcommand{\VFto}{\xrightarrow{\>\mathrm{VF}\> }}

\newcommand{\GHZto}{\xrightarrow[Z]{\>\mathrm{GH}\> }}
\newcommand{\GHWto}{\xrightarrow[W]{\>\mathrm{GH}\> }}
\newcommand{\IFZto}{\xrightarrow[Z]{\;\; \mathrm{F}\;\; }}
\newcommand{\IFWto}{\xrightarrow[W]{\;\;\mathrm{F}\;\; }}

\newcommand{\be}{\begin{equation}}
\newcommand{\ee}{\end{equation}}

\newcommand{\R}{\mathbb{R}}
\newcommand{\M}{\mathbf{M}}

\newcommand{\E}{\mathbb{E}}

\newcommand{\diam}{\operatorname{diam}}

\DeclareMathOperator{\set}{set}

\newcommand{\Lip}{\operatorname{Lip}}

\newcommand{\vol}{\operatorname{Vol}}

\newcommand{\rstr}{\llcorner}

\def\rr{\mathbb{R}}

\def\too{\longrightarrow}

\def\Gr{\mathcal{G}_n(r_0,\gamma,D, \alpha)}

\begin{document}

\title[Convergence of points]{Intrinsic flat convergence of points and applications to stability of the positive mass theorem
}

\author[Huang]{Lan-Hsuan Huang}
\thanks{Huang is partially supported by NSF grants DMS-1452477 and DMS-2005588.}
\address{University of Connecticut}
\email{lan-hsuan.huang@uconn.edu}

\author[Lee]{Dan A. Lee}
\thanks{Lee is partially supported by a PSC CUNY Research Grant.}
\address{CUNY Graduate Center and Queens College}
\email{dan.lee@qc.cuny.edu}

\author[Perales]{Raquel Perales}
\thanks{}
\address{CONACyT Research Fellow at the National Autonomous University of Mexico}
\email{raquel.perales@im.unam.mx}

\date{}


\begin{abstract}
We prove results on intrinsic flat convergence of points---a concept first explored by Sormani in~\cite{Sormani-AA}. In particular, we discuss compatibility with Gromov-Hausdorff convergence of points---a concept first described by Gromov in~\cite{Gromov-poly}. 

We  apply these results to the problem of stability of the positive mass theorem in mathematical relativity. Specifically, we revisit the article~\cite{HLS} on intrinsic flat stability for the case of graphical hypersurfaces of Euclidean space: We are able to fill in some details in the proofs of Theorems 1.4 and Lemma~5.1 of~\cite{HLS} and strengthen some statements. Moreover, in light of an acknowledged error in the proof of Theorem~1.3 of~\cite{HLS},  we provide an alternative proof that extends recent work of~\cite{AP20}. 
\end{abstract}

\maketitle


\section{Introduction}\label{introduction}

Questions concerning  convergence of Riemannian manidolds with lower bounds on  scalar curvature  have attracted increasing attention in the past decade. From mathematical relativity, the question of \emph{stability of the positive mass theorem} asks:  If  a sequence of complete asymptotically flat $n$-dimensional manifolds $(M_j, g_j)$ with nonnegative scalar curvature has ADM masses converging to zero, in what sense must the sequence $(M_j, g_j)$ converge to  Euclidean space? In~\cite{LeeSormani1}, Sormani and the second author observed that convergence  fails in the Gromov-Hausdorff topology in general, but they conjectured that  convergence holds (outside the apparent horizon) in the intrinsic flat topology of Sormani and Wenger~\cite{SorWen2} and established the conjecture in the spherically symmetric case.

There has been much recent progress on applying Sormani-Wenger's intrinsic flat convergence to scalar curvature convergence problems in certain special cases. In particular, the case of Riemannian manifolds that can be embedded as graphical hypersurfaces in Euclidean space has been studied in \cite{HLS, CC, CPKP}. The advantage in the graphical setting is due to an observation of G.~Lam~\cite{Lam-graph} that the scalar curvature of a graphical hypersurface, which can be expressed as a divergence quantity, induces a ``quasi-local mass'' quantity on  level sets of the graphical hypersurface. Further investigation of Lam's quasi-local mass leads to several intriguing properties of the hypersurfaces. For example,  an alternative proof of rigidity of the positive mass theorem in this setting was given in~\cite{Huang-Wu:2013}. From there, the first two authors obtained the stability of the positive mass theorem in Federer-Fleming's flat topology~\cite{Huang-Lee-graph}. In \cite{HLS}, Sormani and the first two authors developed new tools to understand how the results of~\cite{Huang-Lee-graph} relate to the intrinsic flat topology, and some of these tools have been applied by other works on intrinsic flat topology~\cite{Allen-Sormani:2019, Allen-Sormani:2020, CPKP, Allen-Bryden:2019, Allen-Burtscher:2020}. 

Some of the arguments in~\cite{HLS}, especially regarding intrinsic flat convergence of points, were less mature at the time that  \cite{HLS} was written, but the ideas have influenced the study of pointed intrinsic flat convergence. In this note, we clarify those arguments and establish new results on intrinsic flat convergence of points and its compatibility with Gromov-Hausdorff convergence of points. These general results, proven in Section~\ref{section:point-convergence}, are not specific to the graphical hypersurface setting, and we expect them to find further applications. In Sections~\ref{sec:thm1.4} and~\ref{sec:lemma5.1}, we use these results to  flesh out some missing details in the proofs of Theorems 1.4 and Lemma 5.1 of~\cite{HLS}. Separately, we also address an acknowledged error in the proof of Theorem~1.3 of~\cite{HLS} by providing an alternative proof using recent work of B.~Allen and the third author~\cite{AP20}. In doing so, we verify that all of the results of~\cite{HLS} are true.

We thank Christina Sormani for many valuable discussions, Armando Cabrera Pacheco for contributing some ideas of Section~\ref{section:point-convergence},  Christian Ketterer, and Brian Allen.


\section{Point convergence in Gromov-Hausdorff or intrinsic flat sense} \label{section:point-convergence}

We will introduce new vocabulary and notation that will replace some of the less precise language regarding point convergence that was used in~\cite{HLS, Sormani-AA}.  Otherwise, we will use the same notation and definitions as in~\cite{HLS}, with one main exception.

\begin{notation}
Throughout this paper, we use the notation $B(p, r)$ to denote the \textbf{closed} ball of radius $r$ around $p$, and if there is no point specified, then $B(r)$ is just the closed ball of radius $r$ around the origin in Euclidean space. One reason why we choose this convention is that if we regard a \emph{closed} ball $B(p, r)$ in a complete Riemannian manifold as an integral current space $S(p,r)$, then the canonical set of $S(p,r)$, denoted $\set(S(p,r))$, can be identified with $B(p, r)$, whereas this does not work for open balls. The compactness of closed balls in complete Riemannian manifolds is also convenient. 
\end{notation}

Recall that if metric spaces $(X_j, d_j)$ converge to a metric space $(X_\infty,d_\infty)$ in the Gromov-Hausdorff sense, we write $(X_j, d_j)\GHto (X_\infty,d_\infty)$, or perhaps $X_j\GHto X_\infty$ when there is no chance for confusion. Similarly, if integral current spaces $M_j=(X_j, d_j, T_j)$ converge to an integral current space $M_\infty= (X_\infty, d_\infty, T_\infty)$ in the intrinsic flat sense, we write  $M_j \IFto M_\infty$. It is often convenient to see these convergences as occurring within a fixed metric space, so we introduce the following notation:

\begin{defn}
Consider metric spaces $(X_j, d_j)$ and a choice of a separable complete metric space $(Z, d)$ and metric-isometric embedding maps $\varphi_j:X_j \rightarrow Z$, for  $j\in\mathbb{N}\cup\{\infty\}$.

We say that 
\[ X_j \GHZto X_\infty\]
if and only if $\varphi_j(X_j)\to \varphi_\infty(X_\infty)$ in the Hausdorff sense in $Z$.

If $M_j=(X_j, d_j, T_j)$ are $n$-dimensional integral current spaces for  $j\in\mathbb{N}\cup\{\infty\}$, we say that 
\[ M_j \IFZto M_\infty\]
if and only if ${\varphi_j}_{\#}(T_j)\to {\varphi_\infty}_{\#}(T_\infty)$ in the flat sense in $Z$. 

The $Z$ in this notation is intended to indicate that the maps $\varphi_j$ have also been chosen despite not being written down explicitly. 
\end{defn}
In this language, \cite[Theorem 2.3]{Sormani-AA}, sometimes called Gromov's embedding theorem, becomes:
\begin{thm}[Gromov\footnote{The explanation for why this theorem follows from results of~\cite{Gromov-poly} can be found within the proof of~\cite[Theorem 4.2]{SorWen2}.}]
If $(X_j, d_j)$ are compact metric spaces for $j\in\mathbb{N}\cup \{\infty\}$, then $X_j\GHto X_\infty$ if and only if there exists a compact metric space $Z$ (and embedding maps $\varphi_j$) such that $X_j \GHZto X_\infty$.
\end{thm}
We also have the analogous statement for intrinsic flat convergence:
\begin{thm}[{\cite[Theorem 4.2]{SorWen2}}]
If $M_j=(X_j, d_j, T_j)$ are $n$-dimensional integral current spaces for $j\in\mathbb{N}\cup \{\infty\}$, then $M_j\IFto M_\infty$ if and only if there exists a separable complete metric space $Z$ (and embedding maps $\varphi_j$) such that $M_j \IFZto M_\infty$.
\end{thm}

We introduce notation to deal with the concept of convergence of points in the Gromov-Hausdorff or intrinsic flat sense:

\begin{defn}
Consider metric spaces $(X_j, d_j)$ and a choice of a separable complete metric space $(Z, d)$ and metric-isometric embedding maps $\varphi_j:X_j \rightarrow Z$, for  $j\in\mathbb{N}\cup\{\infty\}$.

For points $x_j\in X_j$ for $j\in\mathbb{N}\cup\{\infty\}$, we say that
\[ (X_j, x_j)\GHZto (X_\infty, x_\infty) \]
if and only if $X_j\GHZto X_\infty$ and also $\varphi_j(x_j)\to \varphi_\infty(x_\infty)$ as points in $Z$.

If $M_j=(X_j, d_j, T_j)$ are $n$-dimensional integral current spaces for  $j\in\mathbb{N}\cup\{\infty\}$, then for points $x_j \in X_j$ for $j\in \mathbb{N}$, and $x_\infty\in \overline{X}_\infty$, we say that
\[ (M_j, x_j)\IFZto (M_\infty, x_\infty) \]
if and only if $M_j\IFZto M_\infty$ and also $\varphi_j(x_j)\to \varphi_\infty(x_\infty)$ as points in $Z$. 
\end{defn}

Note that for the second part of the definition, $x_\infty$ need not lie in $X_\infty$, but the definition makes sense since $\varphi_\infty$ extends to the completion $\overline{X}_\infty$. The concept of point convergence in the intrinsic flat sense was first formulated in~\cite[Definition 3.1]{Sormani-AA}, which referred to ``$x_j\in X_j$ converging to $x_\infty\in \overline{X}_\infty$.'' In our language, this means that there exist $Z$ and $\varphi_j$ such that $(M_j, x_j)\IFZto (M_\infty, x_\infty)$. (See also~\cite[Definitions 2.4 and 2.11]{HLS}.)

As alluded to in~\cite[Remark 2.12]{HLS}, if we have both Gromov-Hausdorff and intrinsic flat convergence, we can use the same embedding space $Z$ for both convergences. Here we state this fact explicitly:

\begin{prop}\label{prop:GHIF}
Let $M_j=(X_j, d_j, T_j)$ be compact $n$-dimensional integral current spaces for $j\in\mathbb{N}\cup\{\infty\}$. Then we have both $X_j\GHto X_\infty$ and $M_j\IFto M_\infty$ if and only if there exists a separable complete metric space $Z$ (and embedding maps $\varphi_j$) such that we have both $X_j\GHZto X_\infty$ and $M_{j}\IFZto M_\infty$. 
\end{prop}
This proposition follows from the same reasoning that was used to prove~\cite[Theorem~3.20]{SorWen2}.

We now state some useful facts about point convergence in the intrinsic flat sense that were proved by Sormani.

\begin{lem}[{\cite[Lemma 3.4]{Sormani-AA}}]\label{lem-xinfty} 
Let $M_j=(X_j, d_j, T_j)$ be $n$-dimensional integral current spaces for $j\in \mathbb{N}\cup\{\infty\}$. If $M_j\IFZto M_\infty$, then for any $x_\infty\in \overline{X}_\infty$, there exist points $x_j\in X_j$ such that $(M_j, x_j)\IFZto (M_\infty, x_\infty)$.
\end{lem}

For an integral current space $(X, d, T)$, a point $x\in X$, and $r>0$, we define:
\begin{equation*}
S(x,r) :=  ( \set(T \rstr {B}(x,r) ), d, T\rstr {B}(x,r)).
\end{equation*}

\begin{lem}[{\cite[Lemma 4.1]{Sormani-AA}}]\label{lem-AASorh} 
Let $M_j=(X_j, d_j, T_j)$ be $n$-dimensional integral current spaces for $j\in \mathbb{N}\cup\{\infty\}$. If $(M_j, x_j)\IFZto (M_\infty, x_\infty)$, then there is a subsequence $x_{j_k} \in X_{j_k}$ such that for almost
 every $r>0$, $S(x_{j_k},r)$ and $S(x_\infty,r)$ are integral currents spaces, and 
 \[\left(S(x_{j_k},r), x_{j_k}\right) \IFZto \left( S(x_\infty,r), x_\infty\right),\]
 with embedding maps given by restriction. 
 \end{lem}
Note that the conclusion only holds for a subsequence rather than the full sequence.

\begin{thm}[{\cite[Theorem 7.1]{Sormani-AA}}]\label{thm-AASorh} 
Let $M_j=(X_j, d_j, T_j)$ be $n$-dimensional integral current spaces for $j\in \mathbb{N}\cup\{\infty\}$. Assume that $M_j \IFZto M_\infty$ and that there exist
$\delta >0$,  a function \hbox{$h:(0,\delta) \to (0, \infty)$}, and a sequence $x_j \in X_j$ such that 
for almost every $r \in (0,\delta)$,  
\begin{align}
\label{equation:nonzero}
\liminf_{j \to \infty} d_{\mathcal F}( S(x_j, r),  {\bf 0}) \geq h(r) >0.
\end{align}
Then there exist a subsequence $x_{j_k}$ and a point $x_\infty\in \overline{X}_\infty$ such that 
\[ (M_{j_k}, x_{j_k}) \IFZto (M_\infty, x_\infty).\]
\end{thm}

Our first result concerns compatibility of point convergence $(M_j, x_j)\IFZto (M_\infty, x_\infty)$ with respect to different choices of $Z$, and also compatibility with point convergence in converging subsets. 
Given an integral current space $M=(X, d, T)$ and a subset $V\subset X$, we define 
\[ M\rstr V:=  ( \set(T \rstr V ), d, T\rstr V ).\]
So for example, for $x\in X$ and $r>0$, $S(x, r):= M\rstr {B}(x, r)$.

\begin{thm}\label{theorem:point-convergence}
Let $M_j=(X_j, d_j, T_j)$ be $n$-dimensional integral current spaces for $j\in \mathbb{N}\cup\{\infty\}$, and for $j\in \mathbb{N}$, let $x_j\in V_j\subset X_j$  such that  $M_j\rstr V_j$ is a $n$-dimensional integral current space.
  Assume the following:
\begin{enumerate}
\item  \label{item:subset-converge} 
 $(M_j\rstr V_j, x_j) \IFWto (N_\infty, x_\infty)$ for some integral current space $N_\infty$, some point $x_\infty\in \overline{\set(N_\infty)}$, and some choice of $W$ (and embedding maps).
\item \label{item:metric-ball} There exists $\delta >0$ such that the metric ball ${B}(x_j, \delta)\subset X_j$ is entirely contained  in $V_j$ for all large $j$.

\end{enumerate}
Then for any choice of $Z$ (and embedding maps) such that $M_j \IFZto M_\infty$, there exist a  subsequence $x_{j_k}$ and a point $x'_\infty\in \overline{X}_\infty$ such that 
\[ (M_{j_k}, x_{j_k}) \IFZto (M_\infty, x'_\infty) .\]
\end{thm}

\begin{rmrk}
Note that the conclusion is nontrivial even when $V_j = X_j$, in which case the assumption~\eqref{item:metric-ball} trivially holds. Note that even in this case, it need not be true that $x'_\infty=x_\infty$.
\end{rmrk}

\begin{proof} 
By assumption~\eqref{item:metric-ball}, for all $r\in (0, \delta)$ and all large $j$, we have 
\[ (M_j\rstr V_j)  \rstr {B^{V_j}}(x_j, r) = M_j\rstr {B^{X_j}}(x_j, r),\]
 so we can unambiguously refer to both spaces as $S(x_j, r)$. 
So by Lemma~\ref{lem-AASorh} and assumption~\eqref{item:subset-converge}, there exists a subsequence $x_{j_k}$ such that for almost every $r\in (0, \delta)$, 
 $S(x_{j_k}, r)\IFWto S(x_\infty, r)$, where $S(x_\infty, r)=N_\infty\rstr {B}(x_\infty, r)$.  So for large $k$, we obviously have
\[ d_{\mathcal{F}}(S(x_{j_k}, r), \mathbf{0}) > \tfrac{1}{2} d_{\mathcal{F}}(S(x_\infty, r), \mathbf{0})>0.\]
Taking $h(r):=\tfrac{1}{2} d_{\mathcal{F}}(S(x_\infty, r), \mathbf{0})$, we see that $S(x_{j_k}, r)$ satisfies the hypotheses of Theorem~\ref{thm-AASorh}. The result then follows from applying  
Theorem~\ref{thm-AASorh} to the convergence $M_{j_k} \IFZto M_\infty$ with points $x_{j_k}\in X_{j_k}$.
\end{proof}

Proposition~\ref{prop:GHIF} tells us that if we have both Gromov-Hausdorff and intrinsic flat convergence, it is possible to find a common embedding space $Z$ in which both types of convergence is ``realized.'' 
Theorem~\ref{theorem:GH}\eqref{item:H-convergence} below shows that
if we have intrinsic flat convergence of spaces whose \emph{boundaries} converge in the Gromov-Hausdorff sense, then again, we can see that both types of convergence are ``realized'' in the same embedding space. (Recall that $n$-dimensional intrinsic flat convergence always implies $(n-1)$-dimensional intrinsic flat convergence of the boundaries.) Roughly speaking, the proofs of~\cite[Theorem 1.4 and Lemma 5.1]{HLS} were written in such a way that they assumed that this theorem is true.

\begin{thm}\label{theorem:GH}
Let $M_j=(X_j, d_j, T_j)$ be $n$-dimensional integral current spaces for $j\in \mathbb{N}\cup\{\infty\}$. 
 Assume $M_j\IFZto M_\infty$ and that we can decompose $\partial M_j=\partial_1 M_j + \partial_2 M_j$ such that
$\partial_2 M_j \IFto \bf{0}$. (In other words, some parts of the boundary are negligible in the intrinsic flat limit.) Define $\Sigma_j:=\set(\partial_1 M_j)$ and $\Sigma_\infty:=\set(\partial M_\infty)$.  

The following statements hold:
\begin{enumerate}[(i)]
 \item If  $(M_j, x_j) \IFZto (M_\infty, x_\infty)$,  then $d_\infty(x_\infty, \Sigma_\infty)\ge \displaystyle \limsup_{j\to\infty} d_j(x_j, \Sigma_j)$. \label{item:bound}
 \item If $\Sigma_j$, $\Sigma_\infty$ are compact  and 
   $(\Sigma_j, d_j)\GHto(\Sigma_\infty, d_\infty)$, then  $(\Sigma_j, d_j)\GHZto(\Sigma_\infty, d_\infty)$.\label{item:H-convergence}
\end{enumerate}
\end{thm}

\begin{proof}
We prove \eqref{item:bound}: Assume $(M_j, x_j) \IFZto (M_\infty, x_\infty)$ with embedding maps $\varphi_j$.
For any $y\in \Sigma_\infty$, we will estimate $d_\infty(x_\infty, y)$ from below. Our hypotheses imply that $\partial_1 M_j \IFZto \partial M_\infty$, so we can apply Lemma~\ref{lem-xinfty} to see that there exists $y_j\in \Sigma_j$ such that $\varphi_{j}(y_j) \to\varphi_\infty(y)$ in $Z$. So 
\begin{align*}
	d_\infty(x_\infty, y) &= d_Z(\varphi_\infty(x_\infty), \varphi_\infty( y))=\lim_{j\to \infty} d_Z (\varphi_{j} (x_j), \varphi_j (y_j)) \\
	&= \limsup_{j\to \infty} d_j (x_j, y_j)\ge  \limsup_{j\to \infty}  d_j (x_j , \Sigma_j ). 
\end{align*}
The result follows by taking the infimum over  $y\in \Sigma_\infty$.
 
We prove \eqref{item:H-convergence}: We assume that  $M_j\IFZto M_\infty$ with embedding maps
$\varphi_j$, and also that $\Sigma_j\GHto \Sigma_\infty$.
Suppose, to get a contradiction, that $\varphi_j (\Sigma_j)$ does not converge to  $\varphi_\infty(\Sigma_\infty)$ in the Hausdorff sense in~$Z$. Then there exists $\epsilon>0$ such that one of the following two cases must occur:
\begin{itemize}
\item There exist a subsequence of $\Sigma_j$, still indexed by $j$, and points $z_j \in \Sigma_\infty$ such that 
\begin{align}\label{equation:points}
	d_Z (\varphi_{j}(\Sigma_j), \varphi_\infty(z_j) ) >\epsilon,
\end{align}
\item  There exist a subsequence of $\Sigma_j$, still indexed by $j$, and points $y_{j}\in \Sigma_{j}$ such that 
\begin{align} \label{equation:points2}
	d_Z (\varphi_{j}(y_{j}), \varphi_\infty(\Sigma_{\infty}) ) >\epsilon.
\end{align}
\end{itemize}

We discuss the first case. By compactness of $\Sigma_\infty$, there is a subsequential limit $z_\infty \in \Sigma_\infty$. By Lemma~\ref{lem-xinfty}, there exist points $y_j\in \Sigma_j$ such that $\varphi_j(y_j) \to \varphi_\infty(z_\infty)$ in $Z$, but this contradicts equation~\eqref{equation:points}.

We discuss the second case. Since $\Sigma_j$, $\Sigma_\infty$ are compact and we have both $\Sigma_j\GHto\Sigma_\infty$ and $\partial_1 M_j\IFto \partial M_\infty$, Proposition~\ref{prop:GHIF} tells us that there exist a separable complete metric space $W$ and maps $\psi_j$ such that $\Sigma_j \GHWto \Sigma_\infty$ and $\partial_1 M_j \IFWto \partial M_\infty$. In particular, 
 $d_W (\psi_{j}(y_{j}), \psi_\infty(\Sigma_{\infty}) ) \to 0$. 
 Since $\psi_\infty(\Sigma_{\infty})$ is compact, it follows that there is a subsequence of $\psi_j(y_j)$, which we still index by $j$, that converges to something in $\psi_\infty(\Sigma_{\infty})$. So there exists $y_\infty\in \Sigma_{\infty}$ such that 
 \[ (\partial_1 M_j, y_j)  \IFWto (\partial M_\infty, y_\infty). \]
 So by Theorem~\ref{theorem:point-convergence} applied to the convergence $\partial_1 M_j\IFZto \partial M_\infty$ (and $V_j$ equal to the full space $\Sigma_j$), it follows that there exist a subsequence of $y_j$, still indexed by $j$ and $y'_\infty\in \overline{\Sigma}_\infty=\Sigma_\infty$ such that 
  \[ (\partial_1 M_j, y_j)  \IFZto (\partial M_\infty, y'_\infty). \]
In particular, $\varphi_j(y_j) \to \varphi_\infty(y'_\infty)\in \varphi_\infty(\Sigma_\infty)$, which
 contradicts~\eqref{equation:points2}. 
\end{proof}

Recall that in Lemma~\ref{lem-AASorh}, the conclusion only holds for a subsequence and not necessarily for the original sequence. An elementary theorem of analysis says that  if every subsequence has a subsequence that converges to the same thing, then the original sequence itself must also converge to the same thing. The reason why this principle does not apply to Lemma~\ref{lem-AASorh} is the ``almost every'' part of the conclusion: For any \emph{fixed} radius $r$, we do not know that every subsequence has a converging subsequence. The following proposition explains how we can get around this problem when the integral current spaces are Riemannian manifolds and \emph{intrinsic flat volume convergence} holds. We remark that a result such as this is needed to prove that convergence holds for the original sequence rather than just for a subsequence (even if one only wants the conclusion for almost every $R$). For example, see Theorem~\ref{thm:HLSmain2} below.

\begin{thm}\label{theorem:subsequences}
Let $(M_j, g_j)$ be  Riemannian manifolds with $x_j\in M_j$, for $j\in\mathbb{N}\cup\{\infty\}$. Assume that
every subsequence of $x_{j_k}$ of $x_j$ has a subsequence $x_{{j_k}_\ell}$ such that for almost every $R>0$, 
\[ S(x_{{j_k}_\ell}, R) \VFto S(x_\infty, R).\]

Then for all $R>0$, 
\[ S(x_{j}, R) \VFto S(x_\infty, R).\]
\end{thm}
\begin{rmrk}
In the following proof we note where the Riemannian assumption is used, so that the reader can see when the result applies to more general spaces.
\end{rmrk}
\begin{proof}
First we will prove that $S(x_{j}, R) \IFto S(x_\infty, R)$ for all $R>0$. 
To the contrary, suppose there exists a specific $R>0$ such that $S(x_j, R)$ fails to converge to $S(x_\infty, R)$.
 So there exist $\epsilon>0$ and a subsequence $x_{j_k}$ such that for all $k$,
\begin{equation}\label{indirect}
d_{\mathcal{F}}( S(x_{j_k}, R) , S(x_\infty, R))>\epsilon.
\end{equation}
By our assumption, there exists a subsequence $x_{{j_k}_\ell}$ such that for almost every $r>0$, 
$S(x_{{j_k}_\ell}, r)\VFto S(x_\infty, r)$. We select $R'>R$ 
close enough to $R$ so that 
\begin{align} \label{estimate1}
d_{\mathcal{F}}( S(x_\infty, R), S(x_\infty, R'))&\le
\M\left( \overline{B(x_\infty, R')\smallsetminus  B(x_\infty, R)}\right)\\
&=
 \vol (B(x_\infty, R'))- \vol (B(x_\infty, R))<\epsilon/4.\nonumber
\end{align}
 (This is clearly possible since the limit space is Riemannian.) 
Of course, we can also select $R'$ so that 
 $S(x_{{j_k}_\ell}, R') \VFto S(x_\infty, R')$. 
  In particular, for sufficiently large~$\ell$, we have 
\begin{equation}  \label{estimate2}
 d_{\mathcal{F}}( S(x_{{j_k}_\ell}, R'), S(x_\infty, R')) <\epsilon/4.
 \end{equation}
 By \eqref{indirect}, \eqref{estimate1},  \eqref{estimate2}, and the triangle inequality, we see that for sufficiently large~$\ell$, 
\begin{equation}\label{contradict}
  d_{\mathcal{F}}( S(x_{{j_k}_\ell}, R), S(x_{{j_k}_\ell}, R')) > \epsilon/2,
 \end{equation}
and this is the inequality that we will contradict.

For almost every $r>0$, $\partial S(x_{{j_k}_\ell}, r)\IFto \partial S(x_\infty, r)$, so we also have convergence of slices $\langle M_{{j_k}_\ell}, \rho_{{j_k}_\ell}, r\rangle \IFto  \langle M_{\infty}, \rho_{\infty}, r\rangle $, where $\rho_j$ denotes the distance function to the point $x_j$ in $M_j$, for $j\in\mathbb{N}\cup\{\infty\}$. 
By
 lower semicontinuity of mass under intrinsic flat convergence, 
\[
\M ( \langle M_{{j_k}_\ell}, \rho_{{j_k}_\ell}, r\rangle )
\le \liminf_{\ell\to\infty}
 \M ( \langle M_{\infty}, \rho_{\infty}, r\rangle ).
 \]
Applying the Ambrosio-Kirchheim slicing theorem for the case of a distance function on a Riemannian manifold (in which case it is simply the co-area formula),  and also Fatou's Lemma, 
\begin{align*}
\vol (B(x_\infty, R')) &= \int_0^{R'} \M ( \langle M_{\infty}, \rho_{\infty}, r\rangle )\,dr\\
&\le  \int_0^{R'} \liminf_{\ell\to\infty} \M ( \langle M_{{j_k}_\ell}, \rho_{{j_k}_\ell}, r\rangle ) \,dr\\
&\le  \liminf_{\ell\to\infty}\int_0^{R'}  \M ( \langle M_{{j_k}_\ell}, \rho_{{j_k}_\ell}, r\rangle ) \,dr\\ 
&=   \liminf_{\ell\to\infty}\vol(B(x_{{j_k}_\ell}, R'))\\
&= \vol (B(x_\infty, R')),
\end{align*}
where we use the assumption of volume convergence in the last line. This equality implies that we must actually have the equality 
\[ \M ( \langle M_{\infty}, \rho_{\infty}, r\rangle ) = \liminf_{\ell\to\infty} \M ( \langle M_{{j_k}_\ell}, \rho_{{j_k}_\ell}, r\rangle ) ,\]
for almost every $r< R'$. 

Recall that we do not have good convergence properties for $R$, but we can use the co-area formula and Fatou again, combined with the above equality to obtain:
\begin{align*}
&\limsup_{\ell\to\infty}\left[ \vol (B(x_{{j_k}_\ell}, R'))-\vol  (B(x_{{j_k}_\ell}, R))\right]\\
&\le \limsup_{\ell\to\infty} \vol (B(x_{{j_k}_\ell}, R') )  - \liminf_{\ell\to\infty} \int_0^R \M ( \langle M_{{j_k}_\ell}, \rho_{{j_k}_\ell}, r\rangle ) \,dr \\
&\le \vol (B(x_\infty, R') ) -\int_0^R  \liminf_{\ell\to\infty}\M ( \langle M_{{j_k}_\ell}, \rho_{{j_k}_\ell}, r\rangle ) \,dr \\
&=  \vol (B(x_\infty, R')  )-\int_0^R   \M ( \langle M_{\infty}, \rho_{\infty}, r\rangle ) \,dr \\
&= \vol (B(x_\infty, R'))- \vol (B(x_\infty, R))\\
&<\epsilon/4,
\end{align*}
by assumption~\eqref{estimate1}. Since 
\[ d_{\mathcal{F}}\left( S(x_{{j_k}_\ell}, R'), S(x_{{j_k}_\ell}, R)\right)\le \vol (B(x_{{j_k}_\ell}, R'))-\vol  (B(x_{{j_k}_\ell}, R)),\]
 this contradicts \eqref{contradict}.

Finally, we deal with the possibility that the convergence $S(x_{j}, R) \IFto S(x_\infty, R)$ holds but volume convergence does not. If volume convergence fails, there exist $\epsilon>0$ and a subsequence $x_{j_k}$ such that 
\[
|\vol( B(x_{j_k}, R))  -  \vol(B(x_\infty, R))|>\epsilon.
\]
From here we can use the  same argument as above to get a contradiction in exactly the same way.

\end{proof}


 \section{Application to~\cite{HLS}}\label{section:HLS}

We will briefly recall the main definitions of \cite{HLS}.

\begin{defn}\label{def:Gr}
For $n\ge3$, $r_0, \gamma, D>0$, and $\alpha<0$, define $\Gr$ to be the space of all smooth complete Riemannian manifolds $(M^n,g)$ with nonnegative scalar curvature, possibly with boundary, that admit a smooth Riemannian isometric embedding $\Psi:M\too \E^{n+1}$ such that for some open $U\subset B(r_0/2)\subset \E^n$, the image $\Psi(M)$ is the graph of a function $f\in C^\infty(\E^{n}\smallsetminus \overline{U}) \cap C^0(\E^{n}\smallsetminus {U})$:
\[
\Psi(M)=\left\{(x,f(x)): \,\, x\in \E^n \smallsetminus U\right\}
\]
with empty or minimal boundary:
\[ 
\textrm{either }
\partial M=\emptyset \textrm{ and } U=\emptyset, 
\]
\[ 
\textrm{ or
$f$ is constant on each component of $\partial U$ and }
\lim_{x\to\partial U } |D f(x)|=\infty,
\]
and for almost every $h$, the level set 
\[ 
f^{-1}(h)\subset \E^n
\textrm{ is strictly mean-convex and outward-minimizing,}
\]
where strictly mean-convex means that the mean curvature is strictly positive, and outward-minimizing means that any region of $\E^n$ that contains the region enclosed by $f^{-1}(h)$ must have perimeter at least as large as $\mathcal{H}^{n-1}( f^{-1}(h))$.

In addition we require uniform asymptotic flatness conditions:
\[ 
 |D f| \le \gamma
\textrm{ for }|x|\ge r_0/2  \textrm{ and }
\lim_{x\to\infty} |D f| =0.
\]
If $n\ge 5$, we require that 
$f(x)$ approaches a constant as $x\to\infty$. 
If $n=3$ or $4$, we require that the graph is asymptotically Schwarzschild:\footnote{See~\cite{Huang-Lee-graph} for the definition of the function $S_m$.}
\[
\exists \Lambda, m\in \R \textrm{ such that }
 \left|f(x) - (\Lambda+ S_m(|x|)) \right| \le \gamma |x|^\alpha 
 \textrm{ for } |x|\ge r_0.
\]

For $r\ge r_0$, we define
\[
\Omega(r):=\Psi^{-1}(B(r)\times\rr)\,\,\,
\textrm{ and } \,\,\, \Sigma(r):=\partial\Omega(r)\smallsetminus \partial M,
\]
so that $\Omega(r)$ represents the part of $M$ whose $\Psi$-image lies in the cylinder $B(r)\times\rr$, and $\Sigma(r)$ represents the ``outer'' component of $\partial\Omega(r)$, which is the part of $M$ whose $\Psi$-image lies in the cylindrical shell $\partial B(r)\times\rr$.

Finally, we require a ``bounded depth'' assumption:
 \[ 
\sup\left\{d_M(p,\Sigma(r_0)): p\in \Omega(r_0)\right\}
 \le D.
 \]
\end{defn}

For $n$-dimensional integral current spaces, we  say that $M_j$ converges to $M_\infty$ in the \emph{intrinsic flat volume} sense, or $M_j\VFto M_\infty$, if we have intrinsic flat convergence, $M_j\IFto M_\infty$, as well as $\M (M_j)\to \M (M_\infty)$, where $\M$ denotes the mass of an integral current space (not to be confused with the unrelated concept of ADM mass). Recall that $\M$ is the same thing as $\vol$ for Riemannian spaces.

An equivalent statement of~\cite[Theorem 1.3]{HLS} is the following:

\begin{thm}\label{thm:HLSmain}   
Let $n\ge3$, $r_0, \gamma, D>0$, $\alpha<0$, and $r\ge r_0$. Let $M_j\in\Gr$ and adopt the notation in Definition~\ref{def:Gr} with a $j$-subscript. If the ADM masses of $M_j$ converge to zero, then $\Omega_j(r)$ converges  to the Euclidean ball $B(r)$ in the  intrinsic flat volume sense. That is,
\[
 \Omega_j(r)  \VFto  B(r). 
 \]
  \end{thm}

A.~Cabrera Pacheco, C.~Ketterer, and the third author discovered an error in the proof of this theorem in~\cite{HLS} while researching stability of tori with nonnegative scalar curvature \cite{CPKP}.  See~\cite{HLS-corrigendum}.
 B.~Allen and the third author were able to provide an alternative proof of Theorem~\ref{thm:HLSmain}   under the added assumption that $M$ has no boundary~\cite[Section 7]{AP20}. The alternative proof  is an application of~\cite[Theorem 4.2]{AP20} in conjunction with estimates from~\cite{HLS}. In Section~\ref{sec:main} we will extend that argument to obtain a proof of Theorem~\ref{thm:HLSmain}  in full generality.
 
The following theorem is a consequence of Theorem~\ref{thm:HLSmain}.

  \begin{thm}\label{thm:HLSmain2}   
Let $n\ge3$, $r_0, \gamma, D>0$, and $\alpha<0$. 
Let $M_j\in\Gr$ and adopt the notation in Definition~\ref{def:Gr} with a $j$-subscript. If the ADM masses of $M_j$ converge to zero, then for any sequence of points $p_{j}  \in  \Sigma_{j}(r_0)$ and any $R>0$, 
the geodesic ball $B(p_{j}, R)\subset M_j$ converges to the Euclidean ball $B(R)$ in the  intrinsic flat volume sense. That is, 
\[
 B(p_{j}, R)   \VFto  B(R). 
 \]
 \end{thm}
 
 A slightly weaker version of this theorem appears in~\cite{HLS} as Theorem 1.4. In the course of researching how to use~\cite{CC} to prove an asymptotically hyperbolic version of Theorem~\ref{thm:HLSmain2}, Cabrera Pacheco and the third author identified some parts of the proofs of~\cite[Theorem 1.4 and Lemma 5.1]{HLS} that require further justification. 
In Section~\ref{sec:thm1.4}, we will explain in detail how to apply the results from Section~\ref{section:point-convergence} to prove Theorem~\ref{thm:HLSmain2}, and to be thorough, we also discuss how to apply them to the proof of~\cite[Lemma 5.1]{HLS}  in Section~\ref{sec:lemma5.1}, thereby legitimizing all of the results of~\cite{HLS}.

\subsection{Proof of Theorem~\ref{thm:HLSmain} }\label{sec:main}

Throughout the rest of the paper, we will often abusively refer to regions of Riemannian manifolds as sets, metric spaces, and integral current spaces, depending on what is convenient, as long as there is minimal chance of confusion.

Our task in this section is to adapt the proof from~\cite[Section 7]{AP20} to the case of nontrivial boundary. We will find it convenient to use the following corollary of~\cite[Theorem 4.2]{AP20}, which easily follows from a simple scaling argument and the application of a diffeomorphism:

\begin{thm}[Allen-Perales]\label{thm:AP}
Let $(\Omega_\infty, g_\infty)$ be a smooth compact Riemannian manifold, possibly with boundary, and let $\Omega_j$
be diffeomorphic to $\Omega_\infty$ via $C^1$ diffeomorphisms
\be
\Phi_j: \Omega_\infty \rightarrow \Omega_j,
\ee
such that $\Omega_j$ is equipped with  a continuous metric $g_j$. Assume that this sequence has the following properties:
\be \label{first}
 g_\infty (u, u) < \left(1+ \tfrac{1}{j}\right) g_j (d\Phi_j (u) ,d\Phi_j (u)),
 \ee
for all tangent vectors $u$, 
\be \label{second}
\diam(\Omega_j) \le  L,
\ee
\be \label{third}
\vol (\Omega_j) \to \vol (\Omega_\infty),
\ee
\be \label{fourth}
\vol(\partial \Omega_j) \le A,
\ee
for some constants $L$ and $A$. Further assume that the interior of $(\Omega_\infty, g_\infty)$ is convex. Then $(\Omega_j, g_j)$ converges to $(\Omega_\infty, g_\infty)$ in the intrinsic flat sense.
\end{thm}

\begin{proof}[Proof of Theorem~\ref{thm:HLSmain} ]

Let $M_j$ be a sequence in $\Gr$ whose ADM masses are approaching zero, and let $r\ge r_0$. Adopting the notation in Definition~\ref{def:Gr} with a $j$-subscript, $M_j$ is isometric (via the isometry $\Psi_j$) to the graph of some function $f_j: \mathbb{E}^n\smallsetminus U_j \rightarrow \mathbb{R}$, and $f_j$ is constant on (each component of) $\partial U_j$ and $|Df_j|\to\infty$ at $\partial U_j$.   Recall that $\Omega_j(r)=\Psi_j^{-1}(B(r)\times\mathbb{R})$ is  the subset of $M_j$ corresponding to the part of the graph of $f_j$ lying within the cylinder of radius $r$.

By~\cite[Corollary 4.4]{HLS}, we know that $\vol(\Omega_j(r))\to \vol(B(r))$, so in order to prove the result, it suffices to show that  $\Omega_j(r)$ converges to the Euclidean ball $(B(r), g_{\E})$ in the intrinsic flat sense.

We paraphrase the argument from~\cite{AP20} in the no boundary case as follows: When $M_j$ has no boundary, we can define the diffeomorphism $\Phi_j: B(r)\rightarrow \Omega_j(r) $ to be the ``graphing map''
\be
\Phi_j(x)=\Psi_j^{-1}(x, f_j(x)).
\ee
It is also the inverse of the map $\pi\circ\Psi_j$, where $\pi$ is the projection map to $\E^n$. We claim that this $\Phi_j$ satisfies the hypotheses of Theorem~\ref{thm:AP}. We already mentioned that~\eqref{third} holds, and we know that bounds $L$ and $A$ as in~\eqref{second} and~\eqref{fourth} exist from the proof of~\cite[Theorem 3.1]{HLS}. Inequality~\eqref{first} holds because a graphing map is distance nondecreasing, and finally, the interior of a Euclidean ball is obviously convex. Hence we can apply Theorem~\ref{thm:AP}, with $(\Omega_\infty, g_\infty)=(B(r), g_{\E})$, to conclude that $\Omega_j(r)$ converges to $B(r)$ in the intrinsic flat sense.

We will now generalize this argument to the case of nontrivial boundary. In this case, $\Omega_j(r)$ is not diffeomorphic to $B(r)$, so we cannot apply Theorem~\ref{thm:AP} directly to the sequence $\Omega_j(r)$. Instead we will replace $\Omega_j(r)$ by a new sequence $\tilde\Omega_j$ obtained by ``filling in'' the boundary. 

 Let $L$ be the uniform diameter bound for $\Omega_j(r)$ for all $j$ mentioned above. The space $\tilde\Omega_j$ is obtained from $\Omega_j(r)$ by appending a cylinder  $\partial M_j\times(-L, 0)\cong \partial U_j\times(-L, 0)$  to $\partial M_j\subset \Omega_j(r)$  and then smoothly  ``capping'' the other end of the cylinder. The cap may be regarded as being isometric to the graph of a function on $U_j$, which we will also refer to as $f_j$ for simplicity, where $f_j$ is constant on $\partial U_j$ and $|Df_j|\to\infty$ as we approach $\partial U_j$ from the inside.  (Further details of the capping turn out to be inessential.) See Figure~\ref{figure:extension}.
\begin{figure}[h]    \label{figure:extension}
   \includegraphics[width=0.7\textwidth]{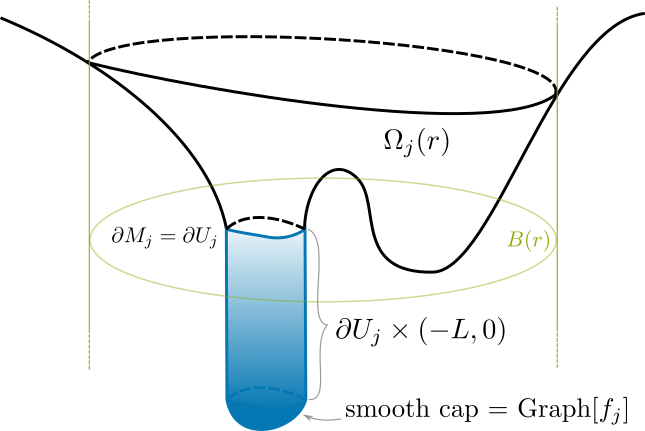}
   \caption{The space $\tilde\Omega_j$ is obtained from $\Omega_j(r)$ by appending a Riemannian cylinder  $\partial M_j\times(-L, 0)\cong \partial U_j\times(-L, 0)$ and a smooth graphical cap.}
\end{figure}

Since the ADM mass of $M_j$ approaches zero, the Penrose inequality~\cite{Lam-graph} implies that $\vol(\partial U_j)=\vol(\partial M_j)\to 0$ also, and then by the isoperimetric inequality, $\vol(U_j)$ approaches zero as well. Since $L$ is fixed, we can certainly choose the cap in such a way that $\tilde\Omega_j \smallsetminus \Omega_j(r)$ has volume approaching zero as $j\to\infty$. Since the cylinder we added has length $L$, it is clear that $\Omega_j(r)$ embeds into $\tilde\Omega_j$ metric isometrically. Therefore it is simple to see that 
\[ d_{\mathcal{F}}(\Omega_j(r), \tilde\Omega_j) \to 0,\]
and hence it only remains to prove that $\tilde\Omega_j$ converges to the Euclidean ball $B(r)$ in the intrinsic flat sense.

The new spaces $\tilde\Omega_j$ are diffeomorphic to $B(r)$, but we must choose the diffeomorphism carefully. The graphing map $\Phi_j(x)= \Psi_j^{-1}(x, f_j(x))$ defines a nice diffeomorphism away from $\partial U_j$, but it needs to be altered in a neighborhood of $\partial U_j$ to obtain a new diffeomorphism $\tilde{\Phi}_j:B(r)\rightarrow \tilde\Omega_j$. The main task is to show that the property~\eqref{first} holds for $\tilde{\Phi}_j$. Intuitively, this is not hard to do:  We just need to ``stretch'' in the directions orthogonal to $\partial U_j$. The stretching can only help $\tilde{\Phi}_j$ to be distance increasing, but the slight change to the tangential directions will introduce a small error term. We describe the details in the following. 

For $\epsilon>0$ sufficiently small, consider a tubular neighborhood of $\partial U_j$ in $\mathbb{E}^n$ diffeomorphic to $\partial U_j\times (-\epsilon, \epsilon)$ via the exponential map $(\theta, \rho) \mapsto (\theta+\rho\nu)$, where $\theta\in \partial U_j$, $\rho\in (-\epsilon, \epsilon)$, and $\nu$ is the outward unit normal to $\partial U_j$ at $\theta$. We will define $\tilde{\Phi}_j$ one piece at a time. For $x\in B(r)$ outside the tubular neighborhood, define $\tilde{\Phi}_j(x)= \Phi_j(x)$. For $0<t<\epsilon$ and $x= \theta+t\nu$, we define
\be
 \tilde{\Phi}_j (x)  = \Phi_j (\theta + \alpha(t)),
 \ee
where $\alpha: (0,\epsilon)\rightarrow (0,\epsilon)$ is an increasing diffeomorphism with the following properties:
\begin{align} 
\lim_{t\to 0^+} \alpha(t)&=0\notag\\
\alpha'(t) &= \frac{2L}{\epsilon}\left[\frac{\partial f_j}{\partial \rho}(\theta+ \alpha(t) \nu) \right]^{-1} &\text{for $t$ near $0$}\label{alpha1}\\
\alpha(t) & = t &\text{for $t$ near $\epsilon$}\label{alpha2}\\
\alpha'(t) &> 2\left[\frac{\partial f_j}{\partial \rho}(\theta+ \alpha(t) \nu) \right]^{-1} &\text{everywhere.}\label{alpha3}
\end{align}
This is easily possible for small $\epsilon$ since $\frac{\partial f_j}{\partial \rho}\to\infty$ as $\rho\to0$. 

Note that on the outer side of the tubular neighborhood, $\tilde{\Phi}_j = \Phi_j\circ\phi_j$, where $\phi_j: \theta+t\nu\mapsto \theta+\alpha(t)\nu$ is a diffeomorphism of the outer side of the tubular neighborhood to itself. The point of~\eqref{alpha1} is to make sure that the diffeomorphism $\tilde{\Phi}_j$ extends from $B(r)\smallsetminus \overline{U}_j$ to the interior of $\Omega_j(r)$ to a diffeomorphism from $B(r)\smallsetminus U_j$ to all of $\Omega_j(r)$ up to the boundary, unlike $\Phi_j$. (The factor of $\frac{2L}{\epsilon}$ is convenient for later.) Equation~\eqref{alpha2} just says that $\phi_j$ smoothly matches the identity map outside the tubular neighborhood. Although $\phi_j$ does some distance contracting in the $\partial_\rho$ direction, inequality~\eqref{alpha3} guarantees that the contracting is counteracted by the stretching by $\Phi_j$, or in other words, we have:
 \be
  g_{\E} ( \partial_\rho, \partial_\rho) \le \tfrac{1}{2}g_j( d\tilde{\Phi}_j (\partial_\rho) ,  d\tilde{\Phi}_j  (\partial_\rho)).
  \ee

Now consider a vector $u$ tangent to a level set of $\rho$. Since we already know that $\Phi_j$ is distance nondecreasing, we only need to understand $\phi_j$, which maps one $\rho$ level set to another. While this map can certainly increase distance, we can keep it controlled by choosing $\epsilon$ small enough. That is, as $\epsilon\to0$, we know that the ratio between induced metrics on parallel surfaces to $\partial U_j$ must approach~$1$. Consequently, one can check that by choosing $\epsilon$ small enough, we have
 \be
 g_{\E} ( u, u) <  (1+\tfrac{1}{j}) g_j( d\tilde{\Phi}_j (u) ,  d\tilde{\Phi}_j  (u)),
 \ee
 for all tangent vectors $u$ in the outer side of the tubular neighborhood. It remains to verify \eqref{first} for a general vector that has a radial part $\partial_\rho$ and a tangential part $u$. While the cross term $g_j( d\tilde{\Phi}_j (\partial_\rho) ,  d\tilde{\Phi}_j  (u))$ could be potentially  large, it is dominated by the radial inner product with a small error of tangential inner product by the Cauchy-Schwarz inequality.

Next, for $t\in (-\tfrac{\epsilon}{2}, 0)$, we define $\tilde{\Phi}_j(\theta+t\nu) = (\theta, \frac{2L}{\epsilon}t) \in \partial U_j \times(-L, 0)$, so that this part of $\tilde{\Phi}_j$ is a diffeomorphism from an inner tubular neighborhood to the cylinder, which is obviously distance nondecreasing. We now see that the factor of $\frac{2L}{\epsilon}$ in~\eqref{alpha1} ensures that these two definitions of $\tilde{\Phi}_j$ match up appropriately along $\partial U_j$ so that $\tilde{\Phi}_j$ is $C^1$. Finally, for $t\in(-\epsilon, -\tfrac{\epsilon}{2})$, we  do something similar to what we did for $t\in(0,\epsilon)$, except now the diffeomorphism $\phi_j$ should be chosen to map the $\rho\in (-\epsilon, -\tfrac{\epsilon}{2})$ part of the tubular neighborhood to the entire $\rho\in(-\epsilon, 0)$ inner side of the tubular neighborhood. It all works out the same way since the graphing function $f_j$ defining the ``cap'' also satisfies $\frac{\partial f_j}{\partial \rho}\to\infty$ as we approach $\partial U_j$ from the inside.

Putting it all together, we have diffeomorphisms $\tilde{\Phi}_j: B(r) \rightarrow \tilde\Omega_j$ satisfying all hypotheses of Theorem~\ref{thm:AP} (noting that we still have a diameter bound for $\tilde\Omega_j$), and the result follows.
\end{proof}


\subsection{Discussion of Theorem~\ref{thm:HLSmain2}} \label{sec:thm1.4} 

  In this section we will explain how Theorem~\ref{thm:HLSmain} implies Theorem~\ref{thm:HLSmain2}, using the results of Section~\ref{section:point-convergence}. 
First, let us briefly summarize the original argument in~\cite{HLS}: Assume $M_j$ as in the hypotheses above, and choose a large $\bar{R}>0$. Theorem~\ref{thm:HLSmain} tells us that  $\Omega_j(r_0+\bar{R})\IFto B(r_0+\bar{R})$.
Starting with a sequence $x_j\in \Sigma_j(r_0)$, we want to extract a subsequential limit in the sense that \hbox{$(\Omega_j(r_0+\bar{R}), x_j)\IFZto (B(r_0+\bar{R}), x_\infty)$} for some $x_\infty$ and $Z$. Then we can invoke Lemma~\ref{lem-xinfty} to obtain the desired result. In~\cite[Lemma 5.1]{HLS} it was shown that $\Sigma_j(r_0)\GHto \partial B(r_0)$, and this implies
one can extract a subsequential limit in the sense that $(\Sigma_j(r_0), x_j)\GHZto (\partial B(r_0), x_\infty)$ for some $x_\infty$ and $Z$. Because of some imprecision of language in~\cite{HLS}, it was implicitly assumed that this is good enough. In this section we will fill in the details:

According to Theorem~\ref{theorem:GH}\eqref{item:H-convergence}, we can find $x_\infty$ and $Z$ such that both 
\[(\Sigma_j(r_0), x_j)\GHZto (\partial B(r_0), x_\infty)\quad\text{and}\quad (\Omega_j(r_0), x_j)\IFZto (B(r_0), x_\infty).\]
This is \emph{almost} what we want, but in order to make the argument completely rigorous, we will use Gromov-Hausdorff convergence of an entire neighborhood of $\Sigma_j(r_0)$ rather than just $\Sigma_j(r_0)$. So we will need the following lemma on 
 convergence of ``coordinate annular regions'' in the exterior part of $M_j$. 
\begin{lem}\label{lemma-annuli}
Assume $M_j$ is a sequence in $\Gr$ with ADM masses approaching zero. For $r>s >r_0 / 2 $, define  $\Omega_j(s,r):=  \overline{\Omega_j(r)  \smallsetminus  \Omega_j(s)} \subset M_j$ and $A(s, r):=\overline{B(r)\smallsetminus B(s)}\subset \E^n$. Then $\Omega_j(s,r)$ converges to $A(s, r)$ in both the Gromov-Hausdorff and intrinsic flat senses. 
\end{lem}
\begin{proof}
We first claim that a subsequence of $\Omega_j(s,r)$ converges to some integral current space $\Omega_\infty(s,r)$ in both the Gromov-Hausdorff and intrinsic flat  senses. Technically, this claim is all that is needed in order to prove Theorem \ref{thm:HLSmain2}. We provide a stronger conclusion in the statement of Lemma~\ref{lemma-annuli} simply because we can.

The claim is proved using the same argument used to prove (the first part of) Lemma~5.1 of \cite{HLS}:  The definition of $\Gr$ implies that the diffeomorphism $\pi \circ  \Psi_j :   \Omega_j(s,r)   \rightarrow  A(s,r)  \times \{0\}$ has a bilipschitz constant which is uniform in $s$, $r$, and $j$, where $\pi$ is the projection map
 from $\mathbb{E}^{n+1}$ onto $\mathbb{E}^n \times \{ 0 \}$.  Then we can apply Theorem A.1 of~\cite{HLS} to obtain the claim. The only complication is that Theorem A.1 of~\cite{HLS} is stated only for integral current spaces \emph{without} boundary, but one can see from its proof that the conclusion will hold as long as the boundary mass is uniformly bounded, as explained in~\cite[Remark 2.22]{Allen-Burtscher:2020}.  Specifically, this can be seen by considering the effect of an extra boundary term on page 294 of~\cite{HLS}, which turns out to be negligible.
  Note that for our desired application, the boundary mass is just $\vol(\partial  \Omega_j(s,r) )$, which we know is uniformly bounded because of the uniform Lipschitz bound on $f_j$.

To obtain the final conclusion, we can use (a much easier version of) the same argument that was used to prove Theorem~1.3 of~\cite{HLS} to see that $\Omega_j(s,r) \IFto A(s, r)$, and hence $\Omega_\infty(s,r)$ must be isometric to $A(s,r)$. Since every subsequence of $\Omega_j(s,r)$ has a subsequence converging in both the Gromov-Hausdorff and intrinsic flat senses to $A(s, r)$, which is independent of choice of subsequence, the original sequence must converge to $A(s,r)$.
 \end{proof}

\begin{proof}[Proof of Theorem \ref{thm:HLSmain2}]
Assume $M_j$ is a sequence in $\Gr$ with ADM masses approaching zero, and let $p_j\in \Sigma_j(r_0)$. 
Choose some large $\bar{R}>1$. Theorem~\ref{thm:HLSmain}  tells us that $\Omega_j(r_0+ \bar{R})\IFZto B(r_0+\bar{R})$ for some choice of $Z$ (and maps), and this is the main ingredient of our proof. 
Our first task is to prove the following:

\underline{Claim:}  There is a subsequence of $p_j$ (still indexed by $j$) such that for almost every $R\in (0, \bar{R}-1)$, we have $B(p_{j}, R) \IFto B(R)$. 
 
Without loss of generality, assume $r_0>2$. Applying Lemma~\ref{lemma-annuli}, we know that $\Omega_j(r_0-1, r_0+1)$ converges in both the Gromov-Hausdorff  and intrinsic flat senses to $A(r_0-1, r_0+1)$. By Proposition~\ref{prop:GHIF}, there exist a separable complete metric space $W$ and embeddings $\psi_j$ such that 
\[ \Omega_j(r_0-1, r_0+1) \GHWto  A(r_0-1, r_0+1)\enspace\text{and}\enspace \Omega_j(r_0-1, r_0+1) \IFWto A(r_0-1, r_0+1).\]
 Since  $\psi_\infty(A(r_0-1, r_0+1))$ is compact, the Hausdorff convergence implies that there is subsequence of $\psi_j(p_j)$, still indexed by $j$, that converges to $\psi_\infty(p_\infty)$ for some $p_\infty\in A(r_0-1, r_0+1)$. Therefore
\[ \left(\Omega_j(r_0-1, r_0+1), p_j\right) \IFWto \left(A(r_0-1, r_0+1), p_\infty\right).\] 

We apply Theorem~\ref{theorem:point-convergence} with $X_j=\Omega_j (r_0+ \bar{R})$ and $V_j =\Omega_j(r_0-1, r_0+1)$ to obtain a subsequence, still indexed by $j$, and a point $p'_\infty$ such  that
\be\label{eqn:converging-points}
 \left( \Omega_j (r_0+ \bar{R}), p_j\right)   \IFZto \left( B(r_0+\bar{R}), p'_\infty\right).
\ee
By Theorem~\ref{theorem:GH}  \eqref{item:bound}, we also have
\[ d_{\E}(p'_\infty, \partial B(r_0 + \bar{R}) ) \ge \limsup_{j\to\infty}  d_j\left(p_j, \Sigma_j (r_0 + \bar{R}) \right)\ge  \bar{R}-1.\]
(Note that we apply Theorem~\ref{theorem:GH} \eqref{item:bound} with $\Sigma_j(r_0+\bar{R})$ as our ``$\partial_1 M_j$'' and $\partial M_j$ as our ``$\partial_2 M_j$,''  the latter of which we know vanishes in the intrinsic flat limit.)

For $R<\bar{R}-1$, we have $B(p_j, R)\subset \Omega_j (r_0+ \bar{R})$ and $B(p'_\infty, R)\subset B(r_0+\bar{R})$, so we can apply Lemma~\ref{lem-AASorh} to~\eqref{eqn:converging-points} to obtain the Claim. 

The proof of Lemma~\ref{lem-AASorh} in~\cite[Lemma~4.1]{Sormani-AA} also shows, by looking at complements of balls rather than the balls themselves, that a further subsequence (still indexed by $j$) satisfies
 $\Omega_{j}(r_0+\bar{R})\smallsetminus B(p_{j}, R)\IFto B(r_0+\bar{R}) \smallsetminus B(p'_\infty, R)$. Using this, the volume convergence argument in~\cite[Theorem 1.4]{HLS} tells us that $\vol(B(p_{j}, R)) \to \vol(B(R))$, and hence  $B(p_{j}, R) \VFto B(R)$ for almost every $R\in (0,\bar{R}-1)$.
 
Finally, since $\bar{R}>1$ was arbitrary, a diagonalization argument shows that there exists a subsequence such that for almost all $R>0$, $B(p_{j}, R) \VFto B(R)$. Then we invoke
Theorem~\ref{theorem:subsequences} to see that for \emph{all} $R>0$, the \emph{original} sequence satisfies $B(p_{j}, R)\VFto B(R)$. 
 \end{proof}

\subsection{Discussion of Lemma 5.1 of~\cite{HLS}} \label{sec:lemma5.1} 

In this section we explain how Theorem~\ref{theorem:GH}\eqref{item:H-convergence} is used in the proof of~\cite[Lemma 5.1]{HLS}. It is only relevant to the second part of~\cite[Lemma 5.1]{HLS}, which says the following:
\begin{lem}
Assume $M_j$ is a sequence in $\Gr$ with ADM masses approaching zero. Then the map $\Psi_\infty:  \set(\Omega_\infty(r))  \rightarrow {B}(r) \times \{0\}$  restricted to $\Sigma_\infty(r):= \set(\partial\Omega_\infty(r))$ is a bilipschitz map onto $\partial{B}(r) \times \{0\}$.
\end{lem}

We briefly recall the construction of $\Omega_\infty(r)$ and $\Psi_\infty$ in~\cite[Theorem 3.1]{HLS}: There exist a subsequence, still indexed by $j$, an integral current space $\Omega_\infty(r)$, and a choice of $Z$, $\varphi_j$  such that $\Omega_j(r)\IFZto\Omega_\infty(r)$. Then
$\Psi_\infty$ was defined so that, after taking an appropriate subsequence, for any $x\in  \set(\Omega_\infty(r))$ and any sequence $x_j\in \Omega_j(r)$, 
\[ \text{if }\quad
(\Omega_j(r), x_j)  \IFZto (\Omega_\infty(r), x),\quad\text{then }\quad  \Psi_\infty(x) = \lim_{j\to\infty} \Psi_j(x_j).\]
In particular, the definition of $\Psi_\infty$ depends on the choice of $Z$ (and choice of subsequence).  We know $\Lip(\Psi_\infty)\le 1$ since each $\Lip(\Psi_j)\le1$, and we know the image of $\Psi_\infty$ lies in ${B}(r) \times \{0\}$ by~\cite[Lemma 4.5]{HLS}. In other words, $\pi\circ\Psi_\infty=\Psi_\infty$, where $\pi$ is the projection map to $\mathbb{E}^n\times\{0\}$.
Finally, since $\Psi_j(\Sigma_j(r)) \subset \partial B(r)\times \R$, we also have $\Psi_\infty( \Sigma_\infty(r))\subset \partial B(r)\times \{0\}$.

\begin{proof}
We will prove that $\Psi_\infty|_{\Sigma_\infty(r)}:\Sigma_\infty(r)\rightarrow \partial{B}(r) \times \{0\}$ is bilipschitz
 by constructing a Lipschitz inverse. We define $\Phi_j : \partial B(r)\times\{0\} \rightarrow \Sigma_j(r)$ to be the inverse of $\pi\circ \Psi_j$, where $\pi$ is the projection map. The first part of~\cite[Lemma 5.1]{HLS} says that $\Sigma_j(r)\GHto\Sigma_\infty(r)$. Since $\partial \Omega_j(r)=\Sigma_j(r) \cup \partial M_j$, and $\partial M_j$ vanishes in the intrinsic flat limit, we can apply Theorem~\ref{theorem:GH}\eqref{item:H-convergence} to see that $\Sigma_j(r)\GHZto\Sigma_\infty(r)$, where $Z$ and $\varphi_j$ are the same metric space and maps that were used to construct~$\Psi_\infty$. (In the original proof of~\cite[Lemma 5.1]{HLS}, it was implicitly assumed that one could use the same  $Z$ and $\varphi_j$ as in the construction of $\Psi_\infty$.)
  
 Since there is a uniform Lipschitz bound for $\Phi_j$, we can extract a subsequence, still indexed by $j$, such that $\Phi_j$ converges to a Lipschitz map
 \[ \Phi_\infty:  \partial B(r)\times\{0\} \rightarrow \Sigma_\infty(r),\]
where $\Phi_\infty$ is defined so that for all $y\in  \partial B(r)\times\{0\}$, 
$(\Sigma_j(r), \Phi_j(y)) \GHZto ( \Sigma_\infty(r), \Phi_\infty(y))$,
or in other words, $(\varphi_\infty\circ\Phi_\infty)(y)=\displaystyle\lim_{j\to\infty} (\varphi_j\circ\Phi_j)(y)$. The proof is completed by showing that $\Phi_\infty$ is the inverse map of $\Psi_\infty|_{\Sigma_\infty(r)}$.

The rest of the argument proceeds as in~\cite[Lemma 5.1]{HLS}. We know that 
\begin{align*}
 \Phi_j\circ \pi \circ \Psi_j &= id: \Sigma_j(r)\rightarrow \Sigma_j(r)\\
\pi \circ \Psi_j \circ \Phi_j &= id: \partial B(r)\times\{0\}\rightarrow \partial B(r)\times\{0\},
\end{align*}
and then the desired result follows from taking limits. We explain this in detail below:

For any $x\in\Sigma_\infty(r)$, Lemma~\ref{lem-xinfty} implies there exists $x_j\in \Sigma_j(r)$ such that the following holds in $Z$:
\begin{align*}
\varphi_\infty(x)&=
\lim_{j\to\infty} \varphi_j(x_j)
=\lim_{j\to\infty} \varphi_j ((\Phi_j\circ \pi \circ \Psi_j) (x_j) )\\
&= 
 (\varphi_\infty\circ\Phi_\infty )\left( \lim_{j\to\infty} (\pi\circ \Psi_j) (x_j) \right)=
 \varphi_\infty (( \Phi_\infty\circ\pi\circ\Psi_\infty)(x)),
\end{align*} 
where we used our definitions of $\Phi_\infty$ and $\Psi_\infty$. So $\Phi_\infty\circ\Psi_\infty=id$ on $\Sigma_\infty(r)$.

Meanwhile, for any $y\in  \partial B(r)\times\{0\}$, by definition of $\Phi_\infty$, 
$\varphi_\infty(\Phi_\infty(y))=\displaystyle \lim_{j\to\infty} \varphi_j(\Phi_j(y))$ in $Z$, and then by definition of $\Psi_\infty$,
$\Psi_\infty(\Phi_\infty(y))= \displaystyle \lim_{j\to\infty} \Psi_j(\Phi_j(y))$ in $\partial B(r)\times \R$. Hence
$( \pi \circ \Psi_\infty\circ\Phi_\infty)(y)
= \lim_{j\to\infty} (\pi \circ \Psi_j \circ \Phi_j )(y)=y$. 
So $\Psi_\infty\circ\Phi_\infty=id$ on $\partial B(r)\times\{0\}$.
\end{proof}

\bibliographystyle{alpha}
\bibliography{2014}
\end{document}